\documentclass[11pt]{amsart}
\usepackage{amsthm,amssymb,mathrsfs,mathtools,empheq}
\usepackage[shortlabels]{enumitem}
\usepackage[T1]{fontenc}
\usepackage{bm}
\usepackage[notref,notcite,final]{showkeys}
\mathtoolsset{showonlyrefs}
\usepackage{fullpage,color}

\newtheorem{mainthm}{Theorem}
\newtheorem{theorem}{Theorem}[section]
\newtheorem*{theorem*}{Theorem}

\newtheorem{lemma}[theorem]{Lemma}

\newtheorem*{proposition*}{Proposition}

\newtheorem*{conjecture*}{Conjecture}

\theoremstyle{definition}
\newtheorem{definition}[theorem]{Definition}
\newtheorem{remark}[theorem]{Remark}

\numberwithin{equation}{section}

\def\bR {\mathbb{R}}
\def\bS {\mathbb{S}}

\def\bZ {\mathbb{Z}}

\def\cE {\mathcal{E}}

\def\cH {\mathcal{H}}

\def\cJ {\mathcal{J}}

\def\scrL{\mathscr{L}}

\def\grad {{\nabla}}

\def\la {\langle}
\def\ra {\rangle}

\newcommand{\wto}{\rightharpoonup}

\newcommand{\wt}[1]{\widetilde{#1}}
\newcommand{\bs}[1]{\boldsymbol{#1}}

\newcommand{\eee}{\mathrm e}

\newcommand{\ud}{\mathrm{\,d}}
\newcommand{\vd}{\mathrm{d}}

\newcommand{\vD}{\mathrm{D}}

\newcommand{\lin}{_{\textsc{l}}}
\newcommand{\linn}{_{n, \textsc{l}}}

\newcommand{\bfd}{{\mathbf d}}

\begin{document}

\title[Two-bubble resolution]{Continuous time soliton resolution \\ for two-bubble equivariant wave maps}
\author{Jacek Jendrej}
\author{Andrew Lawrie}

\begin{abstract}
We consider the energy-critical wave maps equation $\bR^{1+2} \to \bS^2$ in the equivariant case.
We prove that if a wave map decomposes, along a sequence of times, into a superposition
of at most two rescaled harmonic maps (bubbles) and radiation, then such a decomposition holds for continuous time.
If the equivariance degree equals one or two, we deduce, as a consequence of sequential soliton resolution results of C\^ote~\cite{Cote15},
and Jia and Kenig~\cite{JK}, that any topologically trivial equivariant wave map with energy less than four times the energy of the bubble asymptotically decomposes into (at most two) bubbles and radiation.


\end{abstract}

\thanks{J.Jendrej was supported by  ANR-18-CE40-0028 project ESSED.  A. Lawrie was supported by NSF grant DMS-1954455, a Sloan Research Fellowship, and the Solomon Buchsbaum Research Fund}

\maketitle

\section{Introduction} 

\subsection{Setting of the problem}
This paper concerns wave maps  from the Minkowski space $\bR^{1+2}_{t, x}$ into the two-sphere $\bS^2$, with k-equivariant symmetry. These are formal critical points of the Lagrangian action, 
\begin{equation}
\scrL( \Psi)  = \frac{1}{2} \iint_{\bR^{1+2}_{t, x}} \big( {-} |\partial_t \Psi(t, x)|^2 + |\grad \Psi(t, x)|^2 \big) \, \ud x  \ud t, 
\end{equation}
restricted to the class of maps  $\Psi: \bR^{1+2}_{t, x} \to \bS^2 \subset \bR^3$ that take the form, 
\begin{equation}
\Psi(t, r\eee^{i\theta}) = ( \sin \psi (t, r)\cos k\theta , \sin \psi(t, r) \sin  k \theta, \cos \psi(t, r)) \in \bS^2 \subset \bR^3,
\end{equation}
for some fixed $k \in \{1, 2, \ldots\}$. Here $\psi$ is the colatitude measured from the north pole of the sphere,  the metric on $\bS^2$ is given by $\vd s^2 = \vd \psi^2+ \sin^2 \psi\,  \vd \omega^2$,  and $(r, \theta)$ are polar coordinates on $\bR^2$.

Wave maps are called nonlinear $\sigma$-models in the high energy physics literature, see for example~\cite{MS, GeGr17}. They are a canonical example of a geometric wave equation as they generalize the free scalar wave equation to the geometric setting of manifold-valued maps. The $2d$ case we consider is of particular interest, since the static solutions given by finite energy harmonic maps are amongst the simplest examples of topological solitons;  other examples include kinks in scalar field equations,  vortices in Ginzburg-Landau equations, magnetic monopoles, Skyrmions, and Yang-Mills instantons; see~\cite{MS}. 
Wave maps under $k$-equivariant symmetry possess intriguing features from the point of view of nonlinear dynamics, for example, bubbling harmonic maps, multi-soliton solutions, etc.,  in the relatively simple setting of a geometrically natural scalar semilinear wave equation. 
%


The Cauchy problem for $k$-equivariant wave maps is given by
\begin{equation} \label{eq:wmk}
\begin{gathered}
\partial_{t}^2 \psi -   \partial_{r}^2 \psi  - \frac{1}{r}  \partial_r \psi + k^2  \frac{\sin 2 \psi}{2r^2} = 0, \\ 
(\psi(T_0), \partial_t \psi(T_0)) = (\psi_0, \dot \psi_0), \quad T_0 \in \bR.
\end{gathered}
\end{equation}
The conserved energy is 
\begin{equation}
\label{eq:energy} 
E( \bs \psi(t)) := 2 \pi \int_0^\infty \frac{1}{2}  \bigg((\partial_t \psi)^2  + (\partial_r \psi)^2 + k^2 \frac{\sin^2 \psi}{r^2} \bigg)\,r\,\vd r,
\end{equation}
where we have used bold font to denote the vector 
$
\bs \psi (t) := (\psi(t), \partial_t \psi(t)).
$
We will write pairs of functions as  $\bs \phi = (\phi, \dot \phi)$, noting that the notation $\dot \phi$ will not, in general, refer to a time derivative of $\phi$ but rather just to the second component of $\bs \phi$. With this notation~\eqref{eq:wmk} can be rephrased as the Hamiltonian system 
\begin{equation} \label{eq:uham} 
\partial_t \bs \psi(t) =  J \circ \vD E( \bs \psi(t)),\qquad \bs\psi(T_0) = \bs\psi_0,
\end{equation}
where 
\begin{equation} \label{eq:DE}
J = \begin{pmatrix} 0 &1 \\ -1 &0\end{pmatrix}, \quad \vD E( \bs \psi(t))  =  \begin{pmatrix} - \partial_r^2 \psi(t)- r^{-1}\partial_r \psi(t)  + k^2r^{-2}2^{-1}\sin(2\psi(t)) \\ \partial_t \psi(t) \end{pmatrix}.
\end{equation}
We remark that both~\eqref{eq:uham}  and~\eqref{eq:energy} are invariant under the scaling
\begin{equation} \label{eq:uscale} 
(\psi(t, r), \partial_t \psi(t, r))
\mapsto
\big(\psi(t/ \lambda,  r/ \lambda), \lambda^{-1} \partial_t \psi( t/ \lambda, r/ \lambda)\big), \qquad \lambda >0,
\end{equation}
which makes this problem energy-critical.

For $\psi_0: (0, \infty) \to \bR$, we denote
\begin{equation}
\label{eq:energy-pot}
E_p(\psi_0) = 2 \pi \int_0^\infty \frac{1}{2}  \bigg((\partial_r \psi)^2 + k^2 \frac{\sin^2 \psi}{r^2} \bigg)\,r\,\vd r
\end{equation}
the potential part of the energy \eqref{eq:energy}.
 It is easy to check that any $k$-equivariant state $\psi_0$  of finite potential energy must satisfy $\lim_{r \to 0}\psi_0( r) = \ell\pi$ and  $\lim_{r \to \infty}\psi_0(r)=m\pi$  for some $\ell, m \in \bZ$,
 which splits the set of states of finite potential energy into disjoint classes, which we denote $\cH_{\ell, m}$.
 These classes are related to the topological degree of the full map $\Psi(t): \bR^2 \to  \bS^2$:
 if $m - \ell$ is even and $\psi_0 \in \cH_{\ell, m}$, then the corresponding map $\Psi$ is topologically trivial,
 whereas for odd $m - \ell$ we obtain maps of degree $k$.
 
 The sets $\cH_{\ell, m}$ are affine spaces, parallel to the linear space $\cH := \cH_{0, 0}$, which we equip with the norm
 \begin{equation}
 \|\psi_0 \|_{\cH}^2 := \int_0^\infty  \Big((\partial_r \psi_0(r))^2 +  k^2 \frac{\psi_0(r)^2}{r^2}  \Big) \, r \, \vd r.
 \end{equation}
 We denote $L^2 := L^2(r\vd r)$ and $\cE_{\ell, m} := \cH_{\ell, m} \times L^2$ the set of finite energy initial data
 corresponding to the class $\cH_{l, m}$.
 It is natural to consider the Cauchy problem~\eqref{eq:wmk} within a fixed class $\cE_{\ell, m}$.
 The set $\cE:= \cE_{0, 0}$ is a linear space, which comes with the norm
  \begin{equation}
  \| \bs \psi_0 \|_{\cE}^2 := \|\psi_0 \|_{\cH}^2 + \| \dot \psi_0 \|_{L^2}^2 = \int_0^\infty  \Big((\partial_r \psi_0(r))^2 +  k^2 \frac{\psi_0(r)^2}{r^2}  \Big) \, r \, \vd r +  \int_0^\infty \dot \psi_0(r)^2 \, r \, \vd r .
 \end{equation}

The unique (up to scaling and sign change, and adding a multiple of $\pi$)
$k$-equivariant harmonic map is given explicitly by 
\begin{equation} \label{eq:Q-def} 
Q(r) := 2 \arctan (r^k).
\end{equation}
The function $Q$, and its rescaled versions $Q_\lambda(r) := Q(\lambda^{-1}r)$ for $\lambda > 0$,
are minimizers of $E_p$ within the class $\cH_{0, 1}$.
On can compute that $E_p( Q_\lambda ) = 4  \pi k$.
We denote $\bs Q_\lambda := (Q_\lambda, 0)$ the initial data yielding the stationary solution of \eqref{eq:wmk},
$\bs \psi(t) = \bs Q_\lambda$.

Linearizing \eqref{eq:wmk} around $\psi = 0$ leads to the equation
 \begin{equation}
 \label{eq:wmklin}
 \partial_t^2 \psi\lin + L_0 \psi\lin = 0, \qquad \text{where}\quad L_0 := -\partial_r^2 - r^{-1}\partial_r + k^2 r^{-2}.
 \end{equation}
%
 If a solution $\bs \psi$ of \eqref{eq:wmk} converges in the energy space to a solution of \eqref{eq:wmklin}
 as $t \to \infty$, so that the nonlinearity becomes negligible at main order, we say that $\bs \psi$ \emph{scatters}
 in the positive time direction.
 
 \subsection{Sub-threshold theorems, bubbling and soliton resolution}
 The regularity theory for $k$-equivariant wave maps is well understood, see~\cite{CTZduke, CTZcpam, STZ92, STZ94, ShSt00}, and recent research has been focused on the nonlinear dynamics of solutions with large energy. A guiding principle is 
called the \emph{soliton resolution conjecture},  which asserts that every finite energy $k$-equivariant wave map asymptotically decouples into a superposition of finitely many harmonic maps (bubbles) with dynamically separating scaling parameters plus a term capturing the linear radiation.   There has been substantial progress towards this conjecture over the last twenty years. 

Struwe's sequential characterization of singular wave maps~\cite{Struwe} can be viewed as a first step in the direction of soliton resolution. He proved that any wave map that blows up in finite time converges locally along a well chosen sequence of times and a well chosen sequence of scales to a non-constant harmonic map. Struwe's bubbling theorem has an immediate consequence for the regularity theory: every wave map with energy less than that of the ground state harmonic map must be globally regular.  Since only topologically trivial wave maps, and more specifically those in the class $\cE_{\ell,\ell}$, can scatter, one is led to the following formulation of the threshold theorem  proved in~\cite{CKLS1} using the Kenig-Merle road map~\cite{KM06, KM08}: every wave map $\bs\psi \in \cE_{\ell, \ell}$ with $E( \bs  \psi) < 2E( \bs Q)$ must scatter in both time directions. That the threshold is $2 E(\bs Q)$ rather than $E(\bs Q)$ reflects the fact that any $k$-equivariant element of $\cE_{\ell, \ell}$ that develops a bubble must use a least another quantum of  energy $E(\bs Q)$ to connect back to $\ell \pi$; see~\cite{ST2, LO1} for generalizations of these results outside equivariant symmetry. 

Using similar logic, a natural threshold in the class $\cE_{0, 1}$ is $E < 3E( \bs Q)$ since this is the maximal energy level allowing for at most one bubble to form. It was proved in~\cite{CKLS1, CKLS2}, using ideas from Duyckaerts, Kenig, and Merle~\cite{DKM1, DKM2, DKM3, DKM4}, that continuous soliton resolution does hold in this regime. For every global-in-forward-time $1$-equivariant wave map $\bs \psi \in \cE_{0, 1}$ with $E( \bs \psi) < 3E(\bs Q)$ one can find a \emph{continuous} function $\lambda(t)  \ll t$ and a finite energy linear wave  $\bs \psi\lin^*$  so that $\bs \psi(t)$ satisfies, 
\begin{align} 
 \bs \psi(t)  = \bs Q_{\lambda(t)}  + \bs \psi^*\lin(t) + o_{\cE}(1) \, \, \textrm{as} \,\,  t \to \infty .
 \end{align} 
For wave maps in the same class that blow up at a finite time $T_+$ an analogous decomposition holds with $\lambda(t) = o(T_+ - t)$. Following an analogous approach, C\^ote~\cite{Cote15} generalized this result to allow for an arbitrary number of bubbles, but at the cost of only establishing the decomposition along a well chosen \emph{sequence of times}. That is,  given a finite energy wave map in the class $\cE_{0, m}$ one can find a linear wave $\bs\psi\lin^*$,  an integer $J \ge m$, a sequence $t_n \to \infty$, scales $\lambda_{J, n}\ll \lambda_{J-1, n} \ll \dots \ll \lambda_{1, n}$, disjoint subsets $\cJ_-, \cJ_+ \subset \{1, \dots J\}$ with $\# \cJ_- + \#\cJ_+ = J$ and $\#\cJ_+ -\# \cJ_- = m$ so that 
\begin{align} \label{eq:seq} 
\bs \psi(t_n) =   \sum_{j \in \cJ_+}  \bs Q_{\lambda_{j, n}}  - \sum_{k \in \cJ_-}  \bs Q_{\lambda_{k, n}} + \bs \psi\lin^*(t_n) + o_{\cE}(1)  \, \, \textrm{as} \,\,  t \to \infty, 
\end{align} 
with an analogous result if $\bs \psi(t)$ blows up in finite time. Later, Jia and Kenig~\cite{JK} generalized the result from~\cite{Cote15} to $k=2$-equivariant wave maps using some different techniques, and we note that a minor technical observation could be used to obtain such sequential  decompositions in all equivariance classes; see Remark~\ref{r:k}. In this paper we will address the question of how to refine such sequential decompositions to ones that hold continuously in time when the sequential decomposition has at most two bubbles; see~\cite{DKM7, DKM8, DKM9} for a complete resolution of this question for the critical radial focusing NLW in odd dimensions.

\subsection{Statement of the results}
We prove continuous time soliton resolution for a class of initial data not covered in \cite{CKLS1, CKLS2, Cote15, JK}.
\begin{mainthm}
\label{thm:main1}
Fix $k \in \{1, 2\}$.
Let $\bs \psi_0 \in \cE$ such that $E(\bs\psi_0) < 4 E( \bs Q) = 16k\pi$, and let $\bs\psi: [0, T_+) \to \cE$ be the corresponding
solution of \eqref{eq:wmk}. One of the following alternatives holds:
\begin{enumerate}[(i)]
\item (Scattering) $T_+ = \infty$ and $\bs \psi(t)$ scatters as $t \to \infty$,
\item (One-bubble blow-up) $T_+ < \infty$, and there exist $\lambda : [0, T_+) \to (0, \infty)$,
$\bs\psi_0^* \in \cE_{0, 1}$ and $\iota \in \{-1, 1\}$ such that $\lambda(t) \ll T_+ - t$ as $t \to T_+$ and
\begin{equation}
\lim_{t\to T_+}\big\|\bs \psi(t) - \iota\big(\bs Q_{\lambda(t)} - \bs \psi_0^*\big)\big\|_{\cE} = 0.
\end{equation}
\item (Two-bubble blow-up) $T_+ < \infty$, and there exist $\lambda, \mu : [0, T_+) \to (0, \infty)$,
$\bs\psi_0^* \in \cE$ and $\iota \in \{-1, 1\}$ such that $\lambda(t) \ll \mu(t) \ll T_+ - t$ as $t \to T_+$ and
\begin{equation}
\lim_{t\to T_+}\big\|\bs \psi(t) - \iota\big(\bs Q_{\lambda(t)} - \bs Q_{\mu(t)} + \bs \psi_0^*\big)\big\|_{\cE} = 0.
\end{equation}
\item (Global two-bubble) $T_+ = \infty$, and there exist $\lambda, \mu : [0, \infty) \to (0, \infty)$,
a solution $\bs\psi\lin^*: [0, \infty) \to \cE$ of \eqref{eq:wmklin}
and $\iota \in \{-1, 1\}$ such that $\lambda(t) \ll \mu(t) \ll t$ as $t \to \infty$ and
\begin{equation}
\lim_{t\to \infty}\big\|\bs \psi(t) - \iota\big(\bs Q_{\lambda(t)} - \bs Q_{\mu(t)} + \bs \psi\lin^*(t)\big)\big\|_{\cE} = 0.
\end{equation}
\end{enumerate}
\end{mainthm}

\begin{remark} \label{r:k} 
We note that an identical result holds in equivariance classes $k \ge 3$, modulo the completion of the proof of a sequential decompositions as in~\eqref{eq:seq}, which was only carried out in full detail in~\cite{Cote15, JK} in the cases $k =1,2$. The missing technical ingredient in their proofs is the observation that an $L^3_tL^6_x$-type Strichartz estimate (see~\eqref{eq:strich} for the precise norm) holds for the linearization of $k$-equivariant  wave maps equation for every $k \ge 1$ using the estimates proved by Planchon, Stalker, Tahvildar-Zadeh~\cite{PST03b} for the radially symmetric wave equation with inverse square potential. 
\end{remark} 

The theorem stated above will easily follow from the sequential soliton resolution of C\^ote~\cite{Cote15},
and Jia and Kenig~\cite{JK}, once we prove the following result, which is our main contribution.
\begin{mainthm}
 \label{thm:main2}
Fix $k \in \{1, 2, 3, \ldots\}$, $m \in \{0, 1, \ldots\}$, and let $\bs\psi: [0, T_+) \to \cE_{0, m}$ be a solution of \eqref{eq:wmk}.
\begin{enumerate}[1.]
\item (Blow-up case.) Assume $T_+ < \infty$, there exists $\bs \psi^*_0 \in \cE_{0, m}$, and a sequence $t_n \to T_+$ such that
$\lambda_n \ll \mu_n \ll T_+ - t_n$ and
\begin{equation}
\lim_{n\to\infty}\big\|\bs \psi(t_n) - \iota\big(\bs Q_{\lambda_n} - \bs Q_{\mu_n} + \bs \psi^*_0\big)\big\|_\cE = 0.
\end{equation}
Then there exist continuous functions $\lambda, \mu: [T_0, T_+) \to (0, \infty)$ such that
$\lambda(t) \ll \mu(t) \ll T_+ - t$ as $t \to T_+$ and
\begin{equation}
\lim_{t\to T_+}\big\|\bs \psi(t) - \iota\big(\bs Q_{\lambda(t)} - \bs Q_{\mu(t)} + \bs \psi^*_0\big)\big\|_\cE = 0.
\end{equation}
\item (Global case.) Assume $T_+ = \infty$, there exists $\bs \psi\lin^*: [0, \infty) \to \cE$
a solution of \eqref{eq:wmklin}, and a sequence $t_n \to \infty$ such that
$\lambda_n \ll \mu_n \ll t_n$ and
\begin{equation}
\lim_{n\to\infty}\big\|\bs \psi(t_n) - \iota\big(\bs Q_{\lambda_n} - \bs Q_{\mu_n} + \bs \psi\lin^*(t_n)\big)\big\|_\cE = 0.
\end{equation}
Then there exist continuous functions $\lambda, \mu: [T_0, \infty) \to (0, \infty)$ such that
$\lambda(t) \ll \mu(t) \ll t$ as $t \to \infty$ and
\begin{equation}
\lim_{t\to\infty}\big\|\bs \psi(t) - \iota\big(\bs Q_{\lambda(t)} - \bs Q_{\mu(t)} + \bs \psi\lin^*(t)\big)\big\|_\cE = 0.
\end{equation}
\end{enumerate}
\end{mainthm}

\subsection{Comments on the results} 

While the soliton resolution conjecture itself is a qualitative description of the dynamics, it is of central importance to understand which configurations of bubbles and radiation are actually realized by solutions in either the finite time singularity or global-in-time case. 
  Finite-time blow up wave maps with one concentrating bubble were first constructed in a series of influential works by  Krieger, Schlag, Tataru~\cite{KST}, Rodnianski, Sterbenz~\cite{RS}, and Rapha\"el, Rodnianski~\cite{RR}, with the latter work yielding a stable blow-up regime;  see also the recent works~\cite{KrMiao-Duke, KMS}  for stability properties of the solutions from~\cite{KST}, as well as~\cite{JLR1} for a classification of blowup solutions with a given radiation profile, and~\cite{Pil-19, Pil-20} for constructions of various types of solutions with one bubble in infinite time. While these solutions are all constructed within the class $\cE_{0, 1}$, the examples that  blow up in finite time can be smoothly truncated outside the light cone to yield solutions in $\cE$ blowing up in finite time with one bubble, thus the scenario $(ii)$ in Theorem~\ref{thm:main1} is realized.

The first examples of wave maps with two bubbles were constructed by the first author in~\cite{JJ-AJM} in equivariance classes $k \ge 2$.  The solutions in~\cite{JJ-AJM} take the form 
\begin{align} \label{eq:2-bub} 
\bs\psi(t)  = Q_{\lambda(t)} - Q_{\mu(t)} + o_{\cE}(1) \, \, \textrm{as} \,\,  t \to \infty 
 \end{align} 
 with $\lambda(t) \to 0$ and $\mu(t) \to 1$ as $t \to \infty$. The radiation term $ \bs \psi \lin^* = 0$ and thus the solution has threshold energy, i.e., $E( \bs \psi) = 2E( \bs Q)$. In~\cite{JL1} the authors classified the dynamics of every $k$-equivariant wave map with energy $E = 2 E( \bs Q) = 8k\pi$ in \emph{both time directions}, showing, for example, that every such wave map must scatter in at least one time direction. Rodriguez~\cite{R19} proved an analogous result in the case $k=1$ including a construction of a threshold wave map blowing up in finite time in one direction and scattering in the other.  The collision analysis in these papers will play a key role in the proof of Theorem~\ref{thm:main2}; see Section~\ref{ssec:inelastic}. Recently the authors proved that the 2-bubble solution constructed in~\cite{JJ-AJM} is unique and of regularity $H^2$ for classes $k \ge 4$;  see ~\cite{JL2-regularity, JL2-uniqueness}.

C\^ote~\cite{Cote15} observed that a result analogous to Theorem~\ref{thm:main2} holds,
both in the blow-up case and in the global case, if both bubbles have the same sign.
In fact, C\^ote's result allows an arbitrary number of bubbles, all having the same sign.
However, it is worth noting that existence of solutions developing more than one bubble \emph{of the same sign} is unknown.

In the setting of Theorem~\ref{thm:main2}, we know that, at least in the global case,
the set of the initial data satisfying the assumptions
is non-empty, as it contains the two-bubble solution constructed in \cite{JJ-AJM}.
Of course, we expect this set to be much bigger.
Whether the set of initial data satisfying the assumptions of the blow-up case in Theorem~\ref{thm:main2} is non-empty, is unclear to us.
Also, in the case $k = 1$, we do not know if there exist any solutions satisfying the assumptions of Theorem~\ref{thm:main2}.

A natural question is whether our strategy could lead to a proof of soliton resolution for any number of bubbles.
While we believe that studying threshold $N$-bubble solutions for $N \geq 3$ is an interesting topic in itself,
currently it is unknown if this can lead to a proof of soliton resolution in the general case.

\begin{remark}
To be precise, the paper \cite{JJ-AJM} provided a construction for the radial Yang-Mills equation,
which is very similar to equivariant wave maps with $k = 2$.
\end{remark}

\subsection{Comments on the proofs}
\label{ssec:inelastic} 
The following result follows from \cite[Theorem 1.1]{CKLS1}, \cite[Theorem~1.6]{JL1} and \cite[Theorem~1.6]{R19}. 
\begin{theorem}
\label{thm:threshold}
Let $k \in \{1, 2, \ldots\}$.
\begin{enumerate}[1.]
\item
If $\bs\psi$ is a solution of \eqref{eq:wmk} with initial data $\bs \psi_0 \in \cE$ and energy $E(\bs\psi) < 2 E( \bs Q) = 8k\pi$,
then $\psi$ scatters to a linear wave in both time directions.
\item
If $\bs\psi$ is a solution of \eqref{eq:wmk} with initial data $\bs \psi_0 \in \cE$ and  energy $E(\bs\psi) = 2 E( \bs Q) =  8k\pi$,
then $\psi$ scatters to a linear wave in at least one time direction.
\end{enumerate}
\end{theorem}

Settling the threshold case $E(\bs \psi) = 8k\pi$ is essential for our proof of continuous time soliton resolution
of two-bubble wave maps. Indeed, one immediate enemy, when one attempts to deduce continuous time soliton resolution
from sequential soliton resolution, is the possibility of \emph{elastic collisions}.
If colliding solitons could recover their shape after a collision, then one could potentially
encounter the following scenario: the solution approaches a multi-soliton configuration for a sequence of times,
but in between infinitely many collisions take place, so that there is no soliton resolution in continuous time.
The threshold case in Theorem~\ref{thm:threshold} shows in particular that two bubbles cannot collide in an elastic manner.

Transforming the intuition described above into a proof might not be immediate,
see for example the recent preprints \cite{DKM7, DKM8, DKM9}.
In our case, however, we know that threshold two-bubbles not only collide in an inelastic way,
but \emph{scatter} in one time direction. As we demonstrate, this very strong information makes the proof of
continuous time soliton resolution almost immediate.

We stress as well that our proofs crucially use the observation from the earlier works cited above
that \emph{the radiation part} of the solution is extracted for continuous time, and not only for a sequence.
We further comment on this issue at the beginning of Section~\ref{sec:proof} below.

%
\subsection{Acknowledgments}
The importance of the collision problem for the continuous time soliton resolution
was explained to us by Frank Merle.

\section{Preliminaries}
\subsection{Profile decomposition}
Our proof is based on the nonlinear profile decomposition method due to Bahouri and G\'erard~\cite{BG},
and Merle and Vega~\cite{MeVe98}.
For the presentation of this theory in the setting of the equivariant wave maps for arbitrary $k$, see~\cite[Section 2.3]{JL1}.
Here, we only state the relevant results.

For a time interval $I$ and $\psi: I \times (0, \infty) \to \bR$, we define the \emph{Strichartz norm}
\begin{equation}
\label{eq:strich} \|\psi\|_{S(I)} := \bigg(\int_I\bigg(\int_0^\infty\frac{\psi(t, r)^6}{r^3}\ud r\bigg)^\frac 12\bigg)^\frac 13.
\end{equation}
If $I = \bR$, we write $S$ instead of $S(\bR)$.
It is worth noting that a solution $\bs\psi$ of \eqref{eq:wmk} scatters for positive times
if and only if $\|\psi\|_{S([T_0, \infty))} < \infty$.

We also introduce the following notation for the scale change: if $\phi \in \cH_{\ell, m}$,
then $\phi_\lambda(r) := \phi(r / \lambda)$ for all $\lambda > 0$;
if $\bs \phi = (\phi, \dot \phi) \in \cE_{\ell, m}$, then $\bs\phi_\lambda(r) := (\phi(r/\lambda), \lambda^{-1}\phi(r/\lambda))$.
\begin{definition}
\label{def:profil}
We say that a bounded sequence $(\bs\psi_n) \subset \cE$ has a \emph{profile decomposition}
with profiles $\bs U^j_0 \in \cE$ and displacements $(\lambda^j_n, t^j_n)$ if the following conditions are satisfied:
\begin{enumerate}
\item if $j \neq j'$, then $\lim_{n\to\infty} \frac{\lambda^j_n}{\lambda^{j'}_n} + \frac{\lambda^{j'}_n}{\lambda^{j}_n}
+ \frac{|t^{j'}_n - t^j_n|}{\lambda^j_n}= \infty$,
\item if $\bs w_{n,0}^J$ is the \emph{remainder term} defined by
\begin{equation}
\bs \psi_n = \sum_{j=1}^J \bs U\lin^j(-t^j_n / \lambda^j_n)_{\lambda^j_n} + \bs w_{n,0}^J,
\end{equation}
then
\begin{equation}
\lim_{J \to \infty}\limsup_{n\to\infty}\|w\linn^J\|_{S} = 0,
\end{equation}
where $\bs U\lin^j: \bR \to \cE$ and $\bs w\linn^J: \bR \to \cE$ are the solutions of \eqref{eq:wmklin}
such that $\bs U\lin^j(0) = \bs U_0$ and $\bs w\linn^J(0) = \bs w_{n,0}^J$.
\end{enumerate}
\end{definition}
\begin{lemma}[Linear Profile Decomposition]
Every bounded sequence $(\bs\psi_n)\subset \cE$ has a subsequence
which has a profile decomposition.\qed
\end{lemma}
Without loss of generality, upon taking subsequences and modifying the profiles, one can assume that
for all $j$ one of the following holds: $t^j_n = 0$ for all $n$, $\lim_{n\to \infty}t^j_n / \lambda^j_n = \infty$,
$\lim_{n\to \infty}t^j_n / \lambda^j_n = -\infty$.
The \emph{nonlinear profile} $\bs U^j$ associated with the profile $\bs U^j_0$ is a solution of \eqref{eq:wmk}
defined by the condition
\begin{equation}
\lim_{n\to\infty}\|\bs U^j(-t^j_n / \lambda^j_n) - \bs U\linn^j(-t^j_n / \lambda^j_n)\|_\cE = 0.
\end{equation}
\begin{lemma}[Nonlinear Profile Decomposition]
\label{lem:nlprof}
Let $\bs \psi_{n, 0} \in \cE$ be a bounded sequence with a profile decomposition,
and let $\bs U^j$ be the associated nonlinear profiles, with the maximal forward time of existence $T_+(\bs U^j)$.
The following ``Pythagorean formula'' holds:
\begin{equation}
\label{eq:pytha}
\lim_{J\to\infty}\limsup_{n \to \infty}\Big|E(\bs \psi_{n, 0}) - \sum_{j=1}^J E(\bs U^j) - \|\bs w_{n, 0}^J\|_\cE^2\Big| = 0.
\end{equation}
Furthermore, let $s_n \in (0, \infty)$ be any sequence such that for all $j$ and $n$
\begin{equation}
\frac{s_n - t_{n}^j}{\lambda_{n}^j} < T_+(\bs U^j), \quad \limsup_{n \to \infty} \|U^j\|_{S(I_n^j)} <\infty, \quad\text{where }I_n^j := \Big[{-}\frac{t_{n}^j}{\lambda_{n}^j}, \frac{s_n - t_{n}^j}{\lambda_{n}^j}\Big].
\end{equation}
Let $\bs\psi_n(t)$ denote the solution of \eqref{eq:wmk} with initial data $\bs\psi_n(0) = \bs \psi_{n, 0}$.  Then for $n$ large enough  $\bs \psi_n(t)$ exists on the interval $s \in [0, s_n]$ and satisfies, 
\begin{equation}
\limsup_{n \to \infty} \|\psi_n\|_{S([0, s_n])} < \infty.
\end{equation}
Moreover, the following nonlinear profile decomposition holds for all $s \in [0, s_n]$,  
\begin{equation}
\bs\psi_n(s) = \sum_{j=1}^J \bs U^j\Big( \frac{s- t_{n}^j}{\lambda_{n}^j}\Big)_{\lambda_{n}^j} + \bs w\linn^{J}(s)+ \bs g_n^J(s)
\end{equation}
with  $\bs w\linn^J(t)$ as in Definition~\ref{def:profil} and 
\begin{equation}
\label{eq:nlerror}
\lim_{J \to \infty} \limsup_{n\to \infty} \left( \|g_n^J\|_{S([0, s_n])} + \|\bs g_n^J\|_{L^{\infty}([0, s_n]; \cE)} \right) =0.
\end{equation}
The analogous statement holds for sequences $s_n\in (-\infty, 0)$. \qed
\end{lemma}

\subsection{Bubbles and two-bubbles}
In this section, we state a few useful facts about states $\bs \psi_0 \in \cE$ which are close to a two-bubble.

First, we recall the following variational characterization of $Q$ in  $\cH_{0, 1}$ from~\cite{Cote}, which amounts to the coercivity of the energy functional near $Q$. 
\begin{lemma}\emph{\cite[Proposition $2.3$]{Cote}}
\label{l:decQ}
For any $\epsilon > 0$ there exists $\delta > 0$ such that
if $\psi_0 \in \cH_{0, 1}$ and $E_p(\psi_0) \leq 4k\pi + \delta$,
then there exists $\lambda > 0$ such that $\|\psi_0 - Q_\lambda\|_{\cH} \leq \epsilon$. \qed
\end{lemma}
Next, we consider two-bubble configurations.
\begin{definition}
\label{def:ddef}
Given a map $\bs\psi_0 \in \cE$,
we define its proximity $\bfd_+(\bs \psi_0)$ to a positive pure $2$-bubble
and its proximity $\bfd_-(\bs \psi_0)$ to a negative pure $2$-bubble by 
\begin{equation}
\label{eq:ddef} 
\bfd_\pm(\bs \psi_0) := \inf_{\lambda, \mu >0}  \big(  \| \bs\psi_0 \mp (\bs Q_\lambda - \bs Q_\mu) \|_{\cE}^2 + \left( \lambda/\mu \right)^k \big).
\end{equation}
We also set
\begin{equation}
\bfd(\bs \psi_0) := \min(\bfd_+(\bs \psi_0), \bfd_-(\bs \psi_0)).
\end{equation}
\end{definition}

\begin{lemma}
\label{lem:two-bubble-norm}
For any $\epsilon > 0$ there exists $\delta > 0$ such that
if $\bs\psi_0 \in \cE$, $E(\bs\psi_0) \leq 8k\pi + \delta$ and $\|\bs \psi_0\|_{\cE} \geq \delta^{-1}$,
then $\bfd(\bs \psi_0) \leq \epsilon$.
\end{lemma}
\begin{proof}
The result follows from the proof of Lemma~2.13 in \cite{JL1}, with minor modifications left to the Reader.
\end{proof}

\section{Proofs of the theorems}
\label{sec:proof}
Our proof of Theorem~\ref{thm:main2} uses the observation from~\cite{CKLS1, CKLS2, Cote15} 
that the sequential decomposition provides one profile which is independent of the time sequence:
$\bs \psi_0^*$ in the blow-up case and $\bs\psi\lin^*$ in the global case.
Before stating the claim, we introduce the notation
\begin{equation}
\begin{aligned}
\|\bs\psi_0\|_{\cE(r \leq \rho)}^2 = \|(\psi_0, \dot \psi_0)\|_{\cE(r \leq \rho)}^2 &:= \int_0^\rho  \Big(\dot \psi_0(r)^2 + (\partial_r \psi_0(r))^2 +  k^2 \frac{\psi_0(r)^2}{r^2}  \Big) \, r \, \vd r, \\
\|\bs\psi_0\|_{\cE(r \geq \rho)}^2 = \|(\psi_0, \dot \psi_0)\|_{\cE(r \geq \rho)}^2 &:= \int_\rho^\infty  \Big(\dot \psi_0(r)^2 + (\partial_r \psi_0(r))^2 +  k^2 \frac{\psi_0(r)^2}{r^2}  \Big) \, r \, \vd r.
\end{aligned}
\end{equation}
\begin{lemma}\cite[Propositions 5.1 and 5.2]{Cote15}
 \label{lem:zero-profile}
Fix $k \in \{1, 2, 3, \ldots\}$, $\ell, m \in \{0, 1, \ldots\}$, and let $\bs\psi: [0, T_+) \to \cE_{\ell, m}$ be a solution of \eqref{eq:wmk}.
\begin{enumerate}[1.]
\item (Blow-up case.) Assume $T_+ < \infty$.
Then there exists $p \in \bZ$ and $\bs \psi_0^* \in \cE_{p, m}$ such that $\bs\phi(t) := \bs\psi(t) - \bs\psi_0^*$ satisfies
\begin{equation}
\label{eq:zero-profile-1}
\lim_{t\to T_+}\|\bs \phi(t)\|_{\cE(r \geq T_0 - t)} = 0.
\end{equation}
\item (Global case.) Assume $T_+ = \infty$.
Then there exists a solution $\bs\psi\lin^*: \bR \to \cE$ of \eqref{eq:wmklin}
and an increasing function $A: [0, \infty) \to [0, \infty)$ such that $\lim_{t\to \infty} A(t) = \infty$,  and $\bs\phi(t) := \bs\psi(t) - \bs\psi\lin^*(t)$ satisfies
\begin{equation}
\label{eq:zero-profile-2}
\lim_{t\to \infty}\|\bs \phi(t) - (\pi m, 0)\|_{\cE(r \geq t - A(t))} = 0.
\end{equation}
\end{enumerate}\qed
\end{lemma}
\begin{remark}
In the paper \cite{Cote15}, only the cases $k =1$ and $k=2$ were considered, but the proofs, which are identical to the arguments given in Section 5.1 of~\cite{CKLS1} and Section 3.2 of~\cite{CKLS2}, 
are valid for general $k$, using the Strichartz estimates from \cite{PST03b}.
\end{remark}
In the global case, we will also need the following fact.
\begin{lemma}
\label{lem:weakconv}
Let $t_n$ be an increasing sequence such that $\lim_{n\to \infty}t_n = \infty$,
and let $\rho_n$ be a sequence such that $\lim_{n\to\infty}(t_n - \rho_n) = \infty$.
\begin{enumerate}[1.]
\item
If $\bs\phi\lin: \bR\to \cE$ is a solution of \eqref{eq:wmklin}, then $\lim_{n\to\infty}\|\bs\phi\lin(t_n)\|_{\cE(r \leq \rho_n)} = 0$.
\item
Let $\bs\phi_{0, n} \in \cE$ be a bounded sequence of initial data such that $\lim_{n \to \infty}\|\bs \phi_{0, n}\|_{\cE(r \geq \rho_n)} = 0$.
Let $\bs \phi\linn$ be the solution of \eqref{eq:wmklin} corresponding to the initial data $\bs \phi\linn(t_n ) = \bs\phi_{0, n}$.
Then $\bs \phi\linn(0) \wto 0$ as $n \to \infty$.
\end{enumerate}
\end{lemma}
\begin{proof}
The first claim is a well-known property of the linear wave equation, see for example \cite[Proposition 4]{CKS}.

Regarding the second claim, it is sufficient to prove that every subsequence of $\bs \phi\linn(0)$ has a subsequence weakly converging to $0$.
By weak compactness, we only need to show that $\bs \phi\linn(0) \wto \bs\phi_0$ implies $\bs\phi_0 = 0$.

Let $\bs \phi\linn(0) = \bs \phi_0 + \wt{\bs\phi}_{0, n}$, and let $\wt{\bs\phi}\linn$ be the solution of \eqref{eq:wmklin}
corresponding to the initial data $\wt{\bs\phi}\linn(0) = \wt{\bs\phi}_{0, n}$.
Let $\bs\phi\lin$ be the solution of \eqref{eq:wmklin} for the initial data $\bs\phi\lin(0) = \bs\phi_0$,
so that $\bs \phi\linn(t) = \bs \phi\lin(t) + \wt{\bs\phi}\linn(t)$ for all $t \geq 0$.
It is easy to see that \eqref{eq:wmklin} defines a unitary group in $\cE$, hence
\begin{equation}
\la \bs\phi\lin(t_n), \bs\phi_{0,n} - \bs\phi\lin(t_n)\ra_\cE = \la \bs\phi\lin(t_n), \wt{\bs\phi}\linn(t_n)\ra_\cE = \la \bs\phi_0, \wt{\bs\phi}_{0, n}\ra_\cE \to 0, \quad\text{as }n \to \infty,
\end{equation}
where $\la \cdot, \cdot \ra$ denotes the inner product in $\cE$.

By the first part of the lemma, we have $\|\bs\phi\lin(t_n)\|_{\cE(r \leq \rho_n)} \to 0$,
which yields $\la \bs\phi\lin(t_n), \bs\phi_{0,n}\ra_\cE \to 0$, and we obtain
$\|\bs \phi_0\|_\cE^2 = \la \bs\phi\lin(t_n), \bs\phi\lin(t_n)\ra_\cE \to 0$.
\end{proof}

\begin{proof}[Proof of Theorem~\ref{thm:main2}]
In the proof, we use several times the following fact.
If $\bs\phi_n, \bs\psi_n$ are sequences such that $E(\bs\phi_n), E(\bs\psi_n)$ are bounded,
and $\rho_n \in (0, \infty)$ is a sequence such that
\begin{equation}
\lim_{n\to\infty}\|\bs\phi_n\|_{\cE(r \leq \rho_n)} = 0, \qquad \lim_{n\to\infty}\|\bs\psi_n\|_{\cE(r \geq \rho_n)} = 0,
\end{equation}
then
\begin{equation}
\lim_{n\to\infty}\big(E(\bs\phi_n + \bs\psi_n) - E(\bs\phi_n) - E(\bs\psi_n)\big) = 0.
\end{equation}

\textbf{The blow-up case.}
We first prove that it can be assumed without loss of generality that $m = 0$, so that $\bs\psi_0^* \in \cE$.

To see this, consider the solution $\wt{\bs \psi}$ of \eqref{eq:wmk} corresponding to the initial data
$\wt{\bs \psi}(T_0) = \chi\bs \psi(T_0)$, where $\chi$ is a smooth cut-off function such that $\chi(r) = 1$
if $r \leq \frac 12$ and $\chi(r) = 0$ if $r \geq 1$, and $T_+ - \frac 18 < T_0 < T_+$.
By finite speed of propagation, $\wt{\bs \psi}(t, r) = \bs\psi(t, r)$ for all $t \in [T_0, T_+)$ and $r \leq 3/8$
(since an equivariant wave map can only blow up at $r = 0$, it is clear that $\wt \psi$ does not blow up until time $T_+$).
Let $\wt{\bs\psi}_0^* \in \cE$ be given by Lemma~\ref{lem:zero-profile}, so that $\wt{\bs\phi}(t) := \wt{\bs\psi}(t) - \wt{\bs\psi}_0^*$ satisfies \eqref{eq:zero-profile-1}. It follows that $\wt{\bs\psi}_0^*(r) = {\bs\psi}_0^*(r)$ if $r \leq \frac 38$,
implying $\wt{\bs\phi}(t, r) = {\bs\phi}(t, r)$ if $r \leq \frac 38$.
Thus, for any sequence $s_n \to T_+$, $\bfd_\pm(\bs\phi(s_n)) \to 0$ if and only if $\bs\bfd_\pm(\wt{\bs\phi}(s_n))  \to 0$.
We deduce that (sequential or continuous-time) soliton resolution holds for $\bs\psi$ if and only if it holds for $\wt{\bs\psi}$.

We thus assume $m = 0$. Let $\bs \psi_0^* \in \cE$ be given by Lemma~\ref{lem:zero-profile}. We decompose
\begin{equation}
\bs\psi(t) = \bs\phi(t) + \bs \psi_0^*.
\end{equation}
We have $\lim_{n\to \infty}E(\bs \phi(t_n)) = 8 k\pi$, thus $E(\bs\psi_0^*) = E(\bs\psi) - 8k\pi$.
Since $\lim_{t\to T_+}\|\bs\phi(t)\|_{\cE(r \geq T_+ - t)} = 0$ and $\lim_{t\to T_+}\|\bs\psi_0^*\|_{\cE(r \leq T_+ - t)} = 0$,
we obtain $\lim_{t \to T_+} \big(E(\bs \phi(t) + \bs \psi_0^*) - E(\bs\phi(t)) - E(\bs\psi^*_0)\big) = 0$, in other words
\begin{equation}
\lim_{t\to T_+} E(\bs\phi(t)) = 8k\pi.
\end{equation}
Suppose there exists a sequence $\tau_n \to T_+$ such that $\sup_n \|\bs\phi(\tau_n)\|_\cE < \infty$.
Upon extracting a subsequence, we can assume without loss of generality that the sequence $\bs \phi(\tau_n)$ has a profile decomposition.
For $j \in \{1, 2, \ldots\}$, let $\bs U^j$ be the nonlinear profiles, with the corresponding parameters $\lambda_{n}^j$, $t_{n}^j$.
Let $\bs U^0$ be the solution of \eqref{eq:wmk} with the initial data $\bs U^0(0) = \bs \psi_0^*$,
$t_{n}^0 = 0$, $\lambda_{n}^0 = 1$.
Since $\bs \phi(\tau_n) \wto 0$ as $n \to \infty$, see Lemma~\ref{lem:zero-profile},
the sequence $\bs\psi(\tau_n)$ has a profile decomposition
with profiles $\bs U^j$, $j \in \{0, 1, 2, \ldots\}$, and parameters $\lambda_{n}^j, t_{n}^j$.

Thanks to the Pythagorean formula \eqref{eq:pytha}, either there is just one non-zero profile of energy $8k\pi$
and $\bs w_{n, L}^J = \bs w_{n, L}^1 \to 0$ in $\cE$ for all $J \geq 1$,
or all the profiles scatter in both time directions.

\noindent
\textbf{Case 1.} All the profiles $\bs U^j$ scatter in both time directions.
Fix any $0 < T < T_+(\bs U^0)$.
The assumptions of Lemma~\ref{lem:nlprof} are satisfied with $s_n = T$,
which implies that $\bs \psi$ exists on the time interval $[\tau_n, \tau_n + T]$ for all $n$ large enough.
This is in contradiction with the fact that $\bs\phi$ blows up at $t = T_+$.

\noindent
\textbf{Case 2.} There is just one profile $\bs U^1$, and $\bs w\linn^J = 0$ for $J \geq 1$.
By taking a subsequence and adjusting $\bs U^1$, we can assume that $\lim_{n\to \infty}{t_{n}^1}/{\lambda_{n}^1} \in \{-\infty,\infty\}$, or $t_{n}^1 = 0$ for all $n$.

\noindent
\textbf{Case 2.1.} We either have $\lim_{n \to \infty}{t_{n}^1}/{\lambda_{n}^1} = -\infty$,
or $t_{n}^1 = 0$ for all $n$ and $\bs U^1$ scatters in the forward time direction.
In this situation, the same argument as in Case 1 yields a contradiction.

\noindent
\textbf{Case 2.2.} We either have $\lim_{n \to \infty}{t_{n}^1}/{\lambda_{n}^1} = \infty$,
or $t_{n}^1 = 0$ for all $n$ and $\bs U^1$ scatters in the backward time direction.

Let $\bs\psi_n(t) := \bs\psi(\tau_n + t)$.
The assumptions of Lemma~\ref{lem:nlprof} are satisfied with $s_n = -\infty$.
For $t_m$ fixed and $n$ large enough so that $t_m < \tau_n$, we obtain
\begin{equation}
\label{eq:blowup-22}
\bs \psi(t_m) = \bs \psi_n(-(\tau_{n} - t_m)) = \bs U^1\Big( \frac{-(\tau_n - t_m) - t_n^1}{\lambda_n^1} \Big)
+ \bs\psi^*_0 + \bs h_n,
\end{equation}
with $\lim_{n \to \infty}\bs \|\bs h_n\|_\cE = 0$.
Since $\bs U^1$ scatters in the backward time direction, there exists $\epsilon > 0$ such that
$\bfd(\bs U^1(t)) \geq 2\epsilon$ for all $t \leq 0$.
Taking $n \to \infty$ in \eqref{eq:blowup-22}, we obtain
\begin{equation}
\bfd(\bs \psi(t_m) - \bs\psi_0^*) \geq \epsilon.
\end{equation}
This is true for all $m$, with $\epsilon$ independent of $m$, in contradiction with the assumptions of Theorem~\ref{thm:main2}.

Thus, we have proved that $\lim_{t \to T_+} \|\bs \phi(t)\|_\cE = \infty$.
By Lemma~\ref{lem:two-bubble-norm}, $\bs \phi(t)$ converges to a two-bubble in continuous time.

\textbf{The global case.}
The proof is completely analogous. We decompose
\begin{equation}
\bs\psi(t) = \bs\phi(t) + \bs \psi\lin^*(t).
\end{equation}
and we claim that 
\begin{align}  \label{eq:phi-en-lim} 
\lim_{t \to \infty} E( \bs \phi(t) ) = 8k \pi. 
\end{align} 
To see this, note that the limits $\lim_{n\to \infty}E(\bs \phi(t_n)) = 8 k\pi$,
$\lim_{n\to \infty}\|\bs \phi(t_n)\|_{\cE(r \geq t_n - A(t_n))} = 0$,
and  $\lim_{t \to \infty}\|\bs \psi\lin^*(t)\|_{\cE(r \leq t - A(t))} = 0$,
imply that $$\lim_{n\to \infty}E(\bs\psi\lin^*(t_n)) = E(\bs\psi) - 8k\pi.$$
Next, let $\bs\psi^*(t) \in \cE$ denote the solution to~\eqref{eq:wmk} that scatters to $\bs \psi\lin^*(t)$, i.e., 
\begin{align} 
\| \bs\psi^*(t) - \bs\psi\lin^*(t) \|_{\cE} \to 0 \, \, \textrm{as} \, \, t \to \infty,
\end{align} 
 which, together with the previous displayed equation implies that $E( \bs \psi^*) = E( \bs \psi) - 8 k \pi$. Now define $\bs{\wt \phi}(t)$ by  $\bs \psi(t) = \bs{\wt \phi}(t) + \bs \psi^*(t)$.
Since $\lim_{t\to \infty}\|\bs\psi^*(t)\|_{\cE(r \leq t - A(t))} = 0$ and  $\lim_{t\to \infty}\|\bs {\wt \phi}(t)\|_{\cE(r \geq t - A(t))} = 0$,
we obtain $\lim_{t \to \infty} \big(E(\bs{\wt \phi}(t) + \bs \psi^*(t)) - E(\bs{ \wt\phi}(t)) - E(\bs\psi^*)\big) = 0$, or in other words
\begin{equation}
\lim_{t\to \infty} E(\bs\phi(t)) =  \lim_{t \to \infty} E( \bs{\wt \phi}(t))= E(\bs\psi) - E(\bs\psi^*) = 8k\pi, 
\end{equation}
proving~\eqref{eq:phi-en-lim}. 

Suppose there exists a sequence $\tau_n \to \infty$ such that $\sup_n \|\bs\phi(\tau_n)\|_\cE < \infty$.
Upon extracting a subsequence, there exists a profile decomposition of the sequence $\bs \phi(\tau_n)$.
For $j \in \{1, 2, \ldots\}$, let $\bs U^j$ be the nonlinear profiles, with the corresponding parameters $\lambda_{n}^j$, $t_{n}^j$.
Set $\bs U\lin^0 := \bs \psi\lin^*$, $t_{n}^0 = \tau_n$, $\lambda_{n}^0 = 1$.
In order to check that $\bs U\lin^0$ is a profile, we need to verify that $\bs V\linn(0) \wto 0$,
where $\bs V\linn$ is the solution of \eqref{eq:wmklin} corresponding to the initial data $\bs V\linn(\tau_n) = \bs\phi(\tau_n)$.
By Lemma~\ref{lem:zero-profile}, $\|\bs\phi(\tau_n)\|_{\cE(r \geq \tau_n - A(\tau_n))} \to 0$.
Thus Lemma~\ref{lem:weakconv}, applied with $\rho_n := \tau_n - A(\tau_n)$, implies the claim.

Let $\bs U^0$ be the corresponding nonlinear profile, so
that the sequence $\bs\psi(\tau_n)$ has a profile decomposition
with nonlinear profiles $\bs U^j$, $j \in \{0, 1, 2, \ldots\}$, and parameters $\lambda_{n}^j, t_{n}^j$.

Either there is just one profile of energy $8k\pi$
and $\bs w\linn^J = \bs w\linn^1 \to 0$ in $\cE$ for all $J \geq 1$,
or all the profiles scatter in both time directions.

\noindent
\textbf{Case 1.} All the nonlinear profiles $\bs U^j$ scatter in both time directions.
Since the nonlinear profile $\bs U^0$ scatters as $t \to \infty$,
Lemma~\ref{lem:nlprof} yields a contradiction with the fact that $\bs\phi$ does not scatter as $t \to \infty$.

\noindent
\textbf{Case 2.} There is just one profile $\bs U^1$, and $\bs w\linn^J = 0$ for $J \geq 1$.
By taking a subsequence and adjusting $\bs U^1$, we can assume that $\lim_{n\to \infty}{t_{n}^1}/{\lambda_{n}^1} \in \{-\infty,\infty\}$, or $t_{n}^1 = 0$ for all $n$.

\noindent
\textbf{Case 2.1.} We either have $\lim_{n \to \infty}{t_{n}^1}/{\lambda_{n}^1} = -\infty$,
or $t_{n}^1 = 0$ for all $n$ and $\bs U^1$ scatters in the forward time direction.
In this situation, the same argument as in Case 1 yields a contradiction.

\noindent
\textbf{Case 2.2.} We either have $\lim_{n \to \infty}{t_{n}^1}/{\lambda_{n}^1} = \infty$,
or $t_{n}^1 = 0$ for all $n$ and $\bs U^1$ scatters in the backward time direction.

Let $\bs\psi_n(t) := \bs\psi(\tau_n + t)$.
The assumptions of Lemma~\ref{lem:nlprof} are satisfied with $s_n = -\infty$.
For $t_m$ fixed and $n$ large enough so that $t_m < \tau_n$, we obtain
\begin{equation}
\label{eq:global-22}
\bs \psi(t_m) = \bs \psi_n(-(\tau_{n} - t_m)) = \bs U^1\Big( \frac{-(\tau_n - t_m) - t_n^1}{\lambda_n^1} \Big)
+ \bs\psi\lin^*(t_m) + \bs h_n,
\end{equation}
with $\lim_{n \to \infty}\bs \|\bs h_n\|_\cE = 0$.
Since $\bs U^1$ scatters in the backward time direction, there exists $\epsilon > 0$ such that
$\bfd(\bs U^1(t)) \geq 2\epsilon$ for all $t \leq 0$.
Taking $n \to \infty$ in \eqref{eq:global-22}, we obtain
\begin{equation}
\bfd(\bs \psi(t_m) - \bs\psi\lin^*(t_m)) \geq \epsilon.
\end{equation}
This is true for all $m$, with $\epsilon$ independent of $m$, in contradiction with the assumptions of Theorem~\ref{thm:main2}.

Thus, we have proved that $\lim_{t \to T_+} \|\bs \phi(t)\|_\cE = \infty$.
By Lemma~\ref{lem:two-bubble-norm}, $\bs \phi(t)$ converges to a two-bubble in continuous time.
\end{proof}

\begin{proof}[Proof of Theorem~\ref{thm:main1}]
Let $k = 1$. The energy constraint implies that in Theorem~1.1 in \cite{Cote15}, we have $J \leq 2$.
If $J \neq 2$, then Corollary~1.2 from \cite{Cote15} yields the result.
If $J = 2$, then the assumptions of our Theorem~\ref{thm:main2} are satisfied.

The case $k = 2$ is analogous and uses Theorem~1.2 from \cite{JK}.
\end{proof}
\bibliographystyle{plain}
\bibliography{researchbib}

\bigskip
\centerline{\scshape Jacek Jendrej}
\smallskip
{\footnotesize
 \centerline{CNRS and LAGA, Universit\'e Sorbonne Paris Nord}
\centerline{99 av Jean-Baptiste Cl\'ement, 93430 Villetaneuse, France}
\centerline{\email{jendrej@math.univ-paris13.fr}}
} 
\medskip 
\centerline{\scshape Andrew Lawrie}
\smallskip
{\footnotesize
 \centerline{Department of Mathematics, Massachusetts Institute of Technology}
\centerline{77 Massachusetts Ave, 2-267, Cambridge, MA 02139, U.S.A.}
\centerline{\email{alawrie@mit.edu}}
}

\end{document}